\documentclass[letterpaper, 10 pt, conference]{ieeeconf}  % Comment this line out
\IEEEoverridecommandlockouts                              % This command is only
\usepackage[font=small]{caption}
\usepackage{graphicx} % for pdf, bitmapped graphics files
\usepackage{graphics}
\usepackage{epsfig} % for postscript graphics files
\usepackage{mathptmx} % assumes new font selection scheme installed
\usepackage{times} % assumes new font selection scheme installed
\usepackage{amsmath} % assumes amsmath package installed
\usepackage{amssymb}  % assumes amsmath package installed
\usepackage{xcolor}
\usepackage{caption}
\usepackage{subcaption}
\usepackage{url}            % simple URL typesetting
\usepackage{hyperref}       % hyperlinks
\hypersetup{colorlinks=true, linktoc=all, linkcolor=black,}
\usepackage{tabularx}
\usepackage{tabulary}
\usepackage{enumerate}
\usepackage{mathrsfs}
\usepackage{bbm}
\usepackage{hyphenat}
\usepackage{multirow,hhline,color}
\usepackage{dsfont}
\usepackage[sort]{cite}
\hypersetup{breaklinks=true,colorlinks=true,linkcolor=blue,citecolor=blue}
\usepackage{tcolorbox}
\usepackage[noabbrev,capitalise,nameinlink]{cleveref}
\usepackage{color}

\usepackage{algorithm,algorithmicx,algpseudocode}

\algnewcommand{\algorithmicgoto}{\textbf{go to}}%
\algnewcommand{\Goto}[1]{\algorithmicgoto~\ref{#1}}%
\algnewcommand{\LineComment}[1]{\Statex \(\triangleright\) #1}    
\renewcommand{\algorithmicrequire}{\textbf{\textit{Input: }}}
\renewcommand{\algorithmicensure}{\textbf{\textit{Output: }}}

\newcommand{\nnum}{\nonumber}

\DeclareMathOperator*{\diag}{\textit{diag}}
\DeclareMathOperator*{\trace}{\textit{tr}}
\DeclareMathOperator*{\rank}{\textit{rk}}
\DeclareMathOperator*{\argmin}{argmin}

\newcommand{\E}{\mathbb{E}}

\usepackage{accents}
\newcommand{\ubar}[1]{\underaccent{\bar}{#1}}

\newtheorem{theorem}{\bf Theorem}
\newtheorem{lemma}{\bf Lemma}

\newtheorem{assumption}{\bf Assumption}
\newtheorem{claim}{\bf Claim}

\newcommand{\N}{\mathbb{N}}

\newcommand{\R}{\mathbb{R}}

\newcommand{\Ac}{\mathcal{A}}
\newcommand{\Bc}{\mathcal{B}}

\title{\LARGE \bf
Attack-resilient Estimation for Linear Discrete-time Stochastic Systems with  Input and State Constraints
}
\author{Wenbin Wan$^{\dagger}$, Hunmin Kim$^{\dagger}$, Naira Hovakimyan$^{\dagger}$, and Petros G. Voulgaris$^{\ddagger}$% <-this % stops a space
\thanks{This work has been supported by the National Science Foundation (ECCS-1739732 and CMMI-1663460).}
\thanks{$^{\dagger}$Wenbin Wan, Hunmin Kim, and Naira Hovakimyan are with the Department of Mechanical Science and Engineering, University of Illinois at Urbana-Champaign, Urbana, IL 61801, USA.
{\tt\small  \{wenbinw2, hunmin, nhovakim\}@illinois.edu }}%
\thanks{$^{\ddagger}$Petros G. Voulgaris is with the Department of Aerospace Engineering, University of Illinois at Urbana-Champaign, Urbana, IL 61801, USA.
{\tt\small  \{voulgari\}@illinois.edu }}%
}
\begin{document}
\maketitle
\thispagestyle{empty}
\pagestyle{empty}
%%%%%%%%%%%%%%%%%%%%%%%%%%%%%%%%%%%%%%%%%%%%%%%%%%%%%%%%%%%%%%%%%%%%%%%%%%%%%%%%

%%%%%%%%%%%%%%%%%%%%%%%%%%%%%%%%%%%%%%%%%%%%%%%%%%%%%%%%%%%
%%%%%%%%%%%%%%%%%%    Abstract       %%%%%%%%%%%%%%%%%%%%%%
%%%%%%%%%%%%%%%%%%%%%%%%%%%%%%%%%%%%%%%%%%%%%%%%%%%%%%%%%%%
\begin{abstract}
In this paper, an attack-resilient estimation algorithm is presented for linear discrete-time stochastic systems with state and input constraints. It is shown that the state estimation errors of the proposed estimation algorithm are practically exponentially stable. 
\end{abstract}

%%%%%%%%%%%%%%%%%%%%%%%%%%%%%%%%%%%%%%%%%%%%%%%%%%%%%%%%%%%
%%%%%%%%%%%%%%%%%%    Introduction       %%%%%%%%%%%%%%%%%%
%%%%%%%%%%%%%%%%%%%%%%%%%%%%%%%%%%%%%%%%%%%%%%%%%%%%%%%%%%%
\section{Introduction}

Cyber-Physical Systems (CPS) have been of paramount importance in power systems, critical infrastructures, transportation networks and industrial control systems for many decades~\cite{rajkumar2010cyber}. Recent cases of CPS attacks have clearly illustrated the vulnerability of CPS and raised awareness of the security challenges in these systems. 
These include attacks on large-scale systems, such as the StuxNet virus attack on an industrial supervisory control and data acquisition (SCADA) system~\cite{langner2011stuxnet}, German steel mill cyber attack~\cite{lee2014german}, and attacks on modern vehicles~\cite{koscher2010experimental, checkoway2011comprehensive}.

\emph{Literature review.}
Traditionally, cyber-attack detection has been studied by monitoring the cyber-space misbehavior~\cite{raiyn2014survey}. 
With the emergence of CPS, it becomes vitally important to monitor the physical misbehavior as well, because the attacks on CPS always have an impact on physical systems. 
Model-based detection has been intensively studied in recent years.
Attack detection has been formulated as an $\ell_0$/$\ell_\infty$ optimization problem, which is non-deterministic polynomial-time hard (NP-hard) in~\cite{fawzi2014secure,pajic2014robustness,pajic2017attack}. 
A convex relaxation has been studied in~\cite{fawzi2014secure,pajic2017attack}.
On top of this, the worst case estimation error has been analysed in~\cite{pajic2017attack}. 
A residual-based detector has been designed for power systems against false data injection attacks, and the impact of attacks has been analyzed in~\cite{liu2011false}.
Linear algebraic conditions, as well as graph-theoretic conditions for detectability and identifiability have been provided in~\cite{pasqualetti2013attack}.
A multi-rate controller to detect zero-dynamic attacks has been designed in~\cite{jafarnejadsani2018multirate}.
While most of the detection techniques were passive, some papers have studied active detection~\cite{mo2009secure,mo2014detecting}, where the control input is watermarked with a pre-designed scheme that sacrifices optimality.
The attack detection problem has been formulated as a simultaneous estimation problem of the state and the unknown input in~\cite{yong2015resilients}.
The approach has been extended to nonlinear systems in~\cite{kim2017attack}, constrained systems in~\cite{yong2015simultaneous}, and stochastic random set methods in~\cite{forti2016bayesian}.
The aforementioned detection algorithms rely on stochastic thresholds. For accurate detection, a smaller covariance is desired.

To reduce the covariance, the current paper focuses on information aggregation. In particular, we consider inequality state constraints and input constraints.
There is a rich literature on Kalman filter with constraints~\cite{simon2002kalman,julier2007kalman,ko2007state}.
We refer to~\cite{simon2010kalman} for more details for constrained filtering. 
Unknown input estimation algorithm with input constraints is introduced in~\cite{yong2015simultaneous}.
The current paper considers both inequality state and input constraints for unknown input estimation.

\emph{Contribution.}
We design an attack-resilient estimation algorithm given inequality constraints on the states and the attacks.
The proposed algorithm consists of actuator attack estimation and state estimation.
For each step, we design an optimal linear estimator without considering the constraints and then project the estimates onto the constrained space.
We prove that the projection reduces the estimation error, as well as the error covariance.
The practical exponential stability of the estimation error is proved formally.
A numerical simulation on multi-agent robotic system shows the performance of the proposed attack-resilient estimation algorithm.

The paper is organized as follows: Section \ref{pre} introduces and notations, preliminaries on $\chi^2$ test detection and the problem statement.
Section \ref{sec:alsop} describes the high-level idea of the algorithm.
Section \ref{algodetail} gives a detailed algorithm derivation.
Section \ref{analysis} investigates stability analysis of the algorithm, and all the proofs are presented in Appendix for compactness.
Section \ref{simulation} presents a numerical simulation.
Section \ref{conclusion} draws  conclusions.

%%%%%%%%%%%%%%%%%%%%%%%%%%%%%%%%%%%%%%%%%%%%%%%%%%%%%%%%%%%
%%%%%%%%%%%%%%%%%%    Preliminaries       %%%%%%%%%%%%%%%%%
%%%%%%%%%%%%%%%%%%%%%%%%%%%%%%%%%%%%%%%%%%%%%%%%%%%%%%%%%%%
\section{Preliminaries} \label{pre}
This section discusses some preliminary knowledge including notations, motivation, and problem statement.

%%%%%%%%%%%%%%%%%%    Notations       %%%%%%%%%%%%%%%%%%%%
\subsection{Notations}
The following notations are adopted:
We use the subscript $k$ of $x_k$ to denote the time index;
${\mathbb R}^n$ denotes the n-dimensional Euclidean space;
${\mathbb R}^{n \times m}$ denotes the set of all $n \times m$ real matrices;
$A^\top$, $A^{-1}$, $A^\dagger$, $\diag(A)$, $\trace(A)$ and $\rank(A)$ denote the transpose, inverse, Moore-Penrose pseudoinverse, diagonal, trace and rank of matrix $A$, respectively;
$I$ denotes the identity matrix with an appropriate dimension;
$\|\cdot\|$ denotes the standard Euclidean norm for vector or an induced matrix norm;
${\mathbb E}[\,\cdot\,]$ denotes the expectation operator;
$\times$ is used to denote matrix multiplication when the multiplied terms are in different lines.
For a symmetric matrix $S$, $S > 0$ and $S \geq 0$ indicates that $S$ is positive definite and positive semi-definite, respectively. For a vector $a$, $(a)(i)=a(i)$ denotes the $i^{th}$ element in the vector $a$. Finally
$a$, $\hat{a}$, $\tilde{a} \triangleq a - \hat{a}$ denote the true value, estimate and estimation error of $a$. 

%%%%%%%%%%%%%%%%%%    Motivation       %%%%%%%%%%%%%%%%%%%%
\subsection{Motivation}\label{sec:chi}
\subsubsection{$\chi^2$ test for detection}
In attack detection for stochastic systems, the $\chi^2$ test is widely used~\cite{mo2014detecting, teixeira2010cyber}. The $\chi^2$ test can be stated as:\\
\emph{Given a sample $\hat{v}$ of a Gaussian random vector $v$ with unknown mean and known covariance $\Sigma_v$, the $\chi^2$ test provides statistical evidence of whether $v=0$ or not.
The sample is being normalized by $\hat{v}^\top \Sigma_v^{-1} \hat{v}$ and compared with $\chi^2$ test value.
If $\hat{v}^\top \Sigma_v^{-1} \hat{v}>\chi^2$, then we reject the null hypothesis $H_0: v=0$, and accept alternative hypothesis $H_1: v \neq 0$; i.e., there is significant statistical evidence that $v$ is non-zero.
Otherwise, we accept the null hypothesis; i.e., there is no significant evidence that $v$ is non-zero.}

Given a fixed attack input $v \neq 0$ and attack input estimate $\hat{v} \neq 0$, a smaller covariance induces a larger normalized test value $\hat{v}^\top \Sigma_v^{-1} \hat{v}$, which decreases false negative rates.
To reduce the covariance, the minimum variance estimation method is being considered intensively~\cite{kitanidis1987unbiased,gillijns2007unbiased1,gillijns2007unbiased2}. The current paper pursues an optimal filter design technique.

\subsubsection{Constraints}
It has been shown that constraints can be used to further reduce the covariance in optimal filtering; i.e., state constraints in Kalman filter (KF)~\cite{ko2007state,simon2010kalman}, and input constraints in input and state estimation (ISE)~\cite{yong2015simultaneous}.
We consider linear filtering with both input and state constraints to reduce false negative rates in attack detection and to achieve accurate state estimation.
The constraints are induced by unmodeled dynamics and operational processes.
Some of these examples include vision-aided inertial navigation~\cite{mourikis2007multi}, target tracking~\cite{wang2002filtering} and power systems~\cite{yong2015simultaneous,wood2013power}.

%%%%%%%%%%%%%%%%%%    Problem Statement       %%%%%%%%%%%%%%%%%%%%
\subsection{Problem Statement}
Consider the linear time-varying discrete-time stochastic system:
\begin{equation}
\begin{split}
    x_{k+1} &= A_kx_k+B_ku_k+G_{k}d_k+w_k\\
    y_{k} &= C_{k}x_k + v_{k},  \label{eq:sys2}
\end{split}
\end{equation}
where $x_k \in \R^n$, $u_k \in \R^m$, $d_k \in \R^p$ and $y_{k} \in \R^{l}$ are the state, the known input, the unknown actuator attack and sensor measurement, respectively.
Noises $w_k$ and $v_{k}$ are assumed to be independent identically distributed (i.i.d.) Gaussian random variables with zero means and covariances $Q_k \triangleq \E[w_k w_k^\top] \geq 0 $ and $R_{k} \triangleq \E[v_{k} v_{k}^\top] > 0 $ respectively.
Moreover, $v_{k}$ is also uncorrelated with the initial state $x_0$ and process noise $w_k$. We assume that $\rank(C_kG_{k-1})=p$ as in~\cite{darouach1997unbiased,yong2016unified}.

In the cyber-space, digital attack signals could be unconstrained, but their impact on the physical world is restricted by physical and operational constraints (i.e., $d_k$ is constrained).
Any physical constraints and ability limitations on states and actuator attacks are presented by known inequality constraints:
\begin{equation}
    \mathcal{A}_kd_k \leq b_k, \ \mathcal{B}_kx_k \leq c_k. \label{eq:2}
\end{equation}
We assume that the feasible sets of the constraints ${\mathcal A}_k d_k \leq b_k$ and $\Bc_k x_k \leq c_k$ are non-empty.
The vectors $b_k$ and $c_k$, matrices $\mathcal{A}_k, \mathcal{B}_k, A_k, B_k, C_k$ and $G_k$ are known and bounded. The attacker is able to inject any signal $d_k$ that satisfies the constraint in~\eqref{eq:2}.

The estimator design problem, addressed in this paper, can be stated as:
\emph{Given a linear discrete-time stochastic system \eqref{eq:sys2} with constraints on the actuator attack and state \eqref{eq:2}, design an attack-resilient and stable filtering algorithm that simultaneously estimates the system state and actuator attack.}

%%%%%%%%%%%%%%%%%%%%%%%%%%%%%%%%%%%%%%%%%%%%%%%%%%%%%%%%%%%
%%%%%%%%%%%%%%%%%%    Algorithm Design      %%%%%%%%%%%%%%%
%%%%%%%%%%%%%%%%%%%%%%%%%%%%%%%%%%%%%%%%%%%%%%%%%%%%%%%%%%%
\section{Algorithm Design}\label{sec:algod}
In this section, we design an attack-resilient estimation algorithm with  inequality constraints.
The algorithm design is motivated by unknown input estimation~\cite{darouach1997unbiased,yong2016unified,gillijns2007unbiased}, and a projection method for inequality constraint~\cite{yong2015simultaneous,simon2010kalman}.
We design an estimation algorithm as in~\cite{darouach1997unbiased,yong2016unified,gillijns2007unbiased} without considering the constraint, then project the estimates using inequality constraints as in~\cite{yong2015simultaneous, simon2010kalman}.

\begin{figure}[thpb]
\centering
 \includegraphics[width=0.46\textwidth]{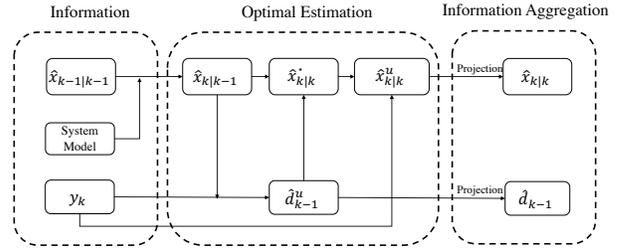}
 \vspace*{-2mm}\caption{The algorithm consists of two parts: optimal estimation and information aggregation. The optimal estimation provides unbiased minimum variance estimates (if the previous state estimate is unbiased) that then be projected in the information aggregation for better estimates.}
\medskip
\label{fig:demo}
\end{figure}

%%%%%%%%%%%%%%%%    Algorithm Statement     %%%%%%%%%%%%%%%
\subsection{Algorithm Statement}\label{sec:alsop}
Given the measurements up to time $k$ and previous state estimate $\hat{x}_{k-1|k-1}$, the proposed algorithm can be summarized as follows:
\begin{enumerate}
\item\emph{Prediction}: 
\begin{align}
    \hat{x}_{k|k-1}=A_{k-1} \hat{x}_{k-1|k-1}+B_{k-1} u_{k-1} \label{eq:statepredict}
\end{align}
\item\emph{Actuator attack estimation}: 
\begin{align}
    &\hat{d}_{k-1}^{u} =M_{k} (y_{k}-C_{k} \hat{x}_{k|k-1})\label{eq:IE1}\\
    &\hat{d}_{k-1} = \argmin\limits_d
    (d-\hat{d}_{k-1}^{u})^\top(P_{k-1}^{d,u})^{-1}(d-\hat{d}_{k-1}^{u})\nnum\\
    &\quad \quad \quad \ {\rm subject\ to \ } \mathcal{A}_{k-1} d \leq b_{k-1}\label{eq:IE2}
\end{align}
\item\emph{Time update}:
\begin{align}
    \hat{x}^\star_{k|k}=\hat{x}_{k|k-1}+G_{k-1} \hat{d}_{k-1}^{u} \label{eq:hat x star}
\end{align}
\item\emph{Measurement update}: 
\begin{align}
    &\hat{x}^u_{k|k}=\hat{x}^\star_{k|k} +L_k(y_{k}-C_{k} \hat{x}^\star_{k|k}) \label{eq:TU2}\\
    &\hat{x}_{k|k} = \argmin\limits_x (x-\hat{x}^{u}_{k|k})^\top (P_{k}^{x,u})^{-1}(x-\hat{x}^{u}_{k|k})\nnum\\
    &\quad \quad \quad \ {\rm subject\ to \ } \mathcal{B}_{k} x \leq c_{k},
    \label{eq:TU3}
\end{align}
\end{enumerate}

Given the previous state estimate $\hat{x}_{k-1|k-1}$, the defender can predict the current state $\hat{x}_{k|k-1}$ under the assumption that the unknown actuator attack is absent (i.e., $d_{k-1}=0$) in~\eqref{eq:statepredict}.
The estimation of the unconstrained actuator attack $\hat{d}_{k-1}^{u}$ can be obtained by observing the difference between the predicted output $C_{k}\hat{x}_{k|k-1}$ and the measured output $y_{k}$ in~\eqref{eq:IE1}, and $M_{k}$ is the filter gain that is chosen to minimize the input error covariances $P_{k}^{d}$. Then, we apply the constraints on the unconstrained actuator attack estimate in~\eqref{eq:IE2} and obtain the constrained actuator attack estimation $\hat{d}_{k-1}$.
The state prediction $\hat{x}_{k|k-1}$ can be updated incorporating the actuator attack estimate $\hat{d}_{k}^{u}$ in~\eqref{eq:hat x star}.
In~\eqref{eq:TU2}, the output $y_{k}$ is used to correct the current state estimate as in KF, where $L_{k}$ is the filter gain that is chosen to minimize the state error covariance $P_{k}^{x,u}$.
The state constraints are applied in~\eqref{eq:TU3} to obtain the constrained state estimation $\hat{x}_{k|k}$. The algorithm is summarized in Fig.~\ref{fig:demo} and presented in Algorithm~\ref{algorithm1}. The detailed algorithm derivation is described in Section~\ref{algodetail}.

%%%%%%%%%%%%%%%%    Algorithm derivation     %%%%%%%%%%%%%%%
\subsection{Algorithm Derivation}\label{algodetail}

\subsubsection{Prediction}
Given the previous state estimate $\hat{x}_{k-1|k-1}$, and the system model~\eqref{eq:sys2}, the current state can be predicted by~\eqref{eq:statepredict} under the assumption that the actuator attack $d_k$ is absent.
Its error covariance matrix is
\begin{align*}
    P_{k|k-1}^x \triangleq \E [\tilde{x}_{k|k-1}\tilde{x}_{k|k-1}^\top]= A_{k-1} P^x_{k-1} A_{k-1}^\top +Q_{k-1},
\end{align*}
where $P_{k}^x \triangleq \E [\tilde{x}_{k|k}\tilde{x}_{k|k} ^\top]$ is the state estimation error covariance matrix.
 %%%%%%%%%%%%%%%%%%%%%%%%%%%%%%%%%%%%%%%%%%%%%%%%%%%%%%%%%%%
%%%%%%%%%%%%%%%%%%    Algorithm      %%%%%%%%%%%%%%%%%%%%%%
%%%%%%%%%%%%%%%%%%%%%%%%%%%%%%%%%%%%%%%%%%%%%%%%%%%%%%%%%%%
\begin{algorithm}[!t] \small
\caption{Attack-resilient Estimation with State and Input Constraint: $\mathcal{A}_kd_k \leq b_k$ and $\mathcal{B}_kx_k \leq c_k$}\label{algorithm1}
\algorithmicrequire $\hat{x}_{k-1|k-1}$; $P^{x}_{k-1}$; \\
\algorithmicensure $\hat{d}_{k-1}$; $P^d_{k-1}$; $\hat{x}_{k|k}$; $P^x_{k}$.
\begin{algorithmic}[1]
%\For {$k =1$ to $N$}
%=====================================================
\LineComment{Prediction}
%=====================================================
\State $\hat{x}_{k|k-1}=A_{k-1} \hat{x}_{k-1|k-1}+B_{k-1} u_{k-1}$;
%=====================================================
\State $P_{k|k-1}^x=A_{k-1} P^x_{k-1} A_{k-1}^\top +Q_{k-1}$;
%=====================================================
\LineComment{Actuator attack estimation}
%=====================================================
\State $\tilde{R}_{k}=C_{k} P_{k|k-1}^x C_{k}^\top+R_{k}$;
%=====================================================
\State $M_{k}=(G_{k-1}^\top C_{k}^\top \tilde{R}_{k}^{-1} C_{k} G_{k-1})^{-1} G_{k-1}^\top C_{k}^\top \tilde{R}^{-1}_{k}$;
%=====================================================
\State $\hat{d}_{k-1}^{u}=M_{k} (y_{k}-C_{k} \hat{x}_{k|k-1})$;
%=====================================================
\State $P^{d,u}_{k-1}=(G_{k-1}^\top C_{k}^\top \tilde{R}_{k}^{-1} C_{k} G_{k-1})^{-1}$;
%=====================================================
\State $P_{k-1}^{xd} =  -P^x_{k-1}A_{k-1}^\top C_{k}^\top M_{k}^\top$
%=====================================================
\State $\hat{d}_{k-1} = \argmin\limits_d (d-\hat{d}_{k-1}^{u})^\top(P_{k-1}^{d,u})^{-1}(d-\hat{d}_{k-1}^{u})$
\Statex \hspace{1cm} subject to $\mathcal{A}_{k-1} d \leq b_{k-1}$;
\State $\bar{\mathcal{A}}_{k-1}$ and $\bar{b}_{k-1}$ corresponding to active set;
%=====================================================
\State $ \gamma_{k-1}^d = P_{k-1}^{d,u} \bar{\mathcal{A}}_{k-1}^\top(\bar{\mathcal{A}}_{k-1} P_{k-1}^{d,u} \bar{\mathcal{A}}_{k-1}^\top)^{-1}$;
%=====================================================
\State $P^{d}_{k-1}=(I-\gamma_{k-1}^d\bar{\mathcal{A}}_{k-1})P^{d,u}_{k-1} (I-\gamma_{k-1}^d\bar{\mathcal{A}}_{k-1}) ^\top$;
%=====================================================
\LineComment{Time update}
%=====================================================
\State $\hat{x}^\star_{k|k}=\hat{x}_{k|k-1}+G_{k-1} \hat{d}_{k-1}^{u}$;
%=====================================================
\State $P^{\star x}_{k}=A_{k-1}P_{k-1}^xA_{k-1}^\top +A_{k-1}P_{k-1}^{xd}G_{k-1}^\top$
\Statex \hspace{1.5cm}$+G_{k-1}(P_{k-1}^{xd})^\top A_{k-1}^\top + G_{k-1}P_{k-1}^{d}G_{k-1}^\top$
\Statex \hspace{1.5cm}$-G_{k-1} M_{k}C_{k}Q_{k-1}-Q_{k-1}C_{k-1}^\top M_{k}^\top G_{k-1}^\top+Q_{k-1}$;
%=====================================================
\State $\tilde{R}^\star_{k}=C_{k} P^{\star x}_{k} C_{k}^\top +R_{k} -C_{k} G_{k-1} M_{k}R_{k}-R_{k} M_{k}^\top G_{k-1}^\top C_{k}^\top$;
%=====================================================
\LineComment{Measurement update}
%=====================================================
\State $L_k=(P^{\star x}_{k} C_{k}^\top - G_{k-1} M_{k}  R_{k})\tilde{R}^{\star \dagger}_{k}$; 
%=====================================================
\State $\hat{x}^{u}_{k|k}=\hat{x}^\star_{k|k} +L_k(y_{k}-C_{k} \hat{x}^\star_{k|k})$;
%=====================================================
\State $P^{x,u}_{k}= (I-L_k C_{k})G_{k-1} M_{k}R_{k}L_k^\top+ L_k R_{k} M_{k}^\top G_{k-1}^\top (I-L_k C_{k})^\top$
\Statex \hspace{1.5cm} $+(I-L_k C_{k}) P^{\star x}_{k} (I-L_k C_{k})^\top+L_k R_{k} L_k^\top$;
%=====================================================
\State $\hat{x}_{k|k} = \argmin\limits_x (x-\hat{x}^{u}_{k|k})^\top (P_{k}^{x,u})^{-1}(x-\hat{x}^{u}_{k|k})$
\Statex \hspace{1cm} subject to $\mathcal{B}_{k} x \leq c_{k}$;
\State $\bar{\mathcal{B}}_{k}$ and $\bar{c}_{k}$ corresponding to active set;
%=====================================================
\State $ \gamma_{k}^x = P_{k}^{x,u} \bar{\mathcal{B}}_{k}^\top(\bar{\mathcal{B}}_{k} P_{k}^{x,u} \bar{\mathcal{B}}_{k}^\top)^{-1}$;
%=====================================================
\State $P^{x}_{k} = (I-\gamma_{k}^x \bar{\mathcal{B}}_{k})P^{x,u}_{k} (I-\gamma_{k}^x \bar{\mathcal{B}}_{k}) ^\top$;
%=====================================================
\end{algorithmic}
\end{algorithm}
%%%%%%%%%%%%%%%%%%%%%%%%%%%%%%%%%%%%%%%%%%%%%%%%%%%%%%%%%%%
%%%%%%%%%%%%%%%%%%    Algorithm End      %%%%%%%%%%%%%%%%%%
%%%%%%%%%%%%%%%%%%%%%%%%%%%%%%%%%%%%%%%%%%%%%%%%%%%%%%%%%%%
\subsubsection{Actuator attack estimation} \label{Actuator attack estimation}
The linear actuator attack estimator in~\eqref{eq:IE1} utilizes the difference between the measured output $y_k$ and the predicted output $C_k\hat{x}_{k|k-1}$.
Substituting~\eqref{eq:sys2} and~\eqref{eq:statepredict} into \eqref{eq:IE1}, we have
\begin{align*}
     \hat{d}_{k-1}^{u} = & M_{k}(C_{k}A_{k-1} \tilde{x}_{k-1|k-1} + C_{k}G_{k-1}d_{k-1}
     + C_{k}w_{k-1} + v_{k}),
\end{align*}
which is a linear function of the actuator attack $d_k$.
Applying the method of least squares from~\cite{sayed2003fundamentals}, which gives linear minimum-variance unbiased estimates,  we can get the optimal gain in actuator attack estimation:
\begin{align*}
   M_{k} = (G_{k-1}^\top C_{k}^\top \tilde{R}_{k}^{-1} C_{k} G_{k-1})^{-1} G_{k-1}^\top  C_{k}^\top \tilde{R}_{k}^{-1},
\end{align*}
where $\tilde{R}_{k} \triangleq C_{k} P_{k|k-1}^x C_{k} + R_{k}$.
It error covariance matrix is found by
\begin{align*}
    P_{k-1}^d &= M_{k}\tilde{R}_{k}M_{k}^\top=(G_{k-1}^\top C_{k}^\top \tilde{R}_{k}^{-1} C_{k} G_{k-1})^{-1}.
\end{align*}
We are now in the position to apply the constraint in~\eqref{eq:2}. The problem is formulated as the constrained convex optimization problem:
\begin{align}
\begin{split}
    \hat{d}_{k-1} &= \argmin\limits_d (d-\hat{d}_{k-1}^{u})^\top W^d_{k-1}(d-\hat{d}_{k-1}^{u}) \\
   & \text{subject to}\  \mathcal{A}_{k-1} d \leq b_{k-1}, \label{opt ineq}
\end{split}
\end{align}
where $W^d_{k-1}$ can be any positive definite symmetric weighting matrix.
In the current paper, we choose $W_{k-1}^d = (P_{k-1}^{d,u})^{-1}$ which results in the smallest error covariance as shown in~\cite{simon2002kalman}.
From Karush-Kuhn-Tucker (KKT) conditions of optimality, we can find the corresponding active constraints.
We denote by $\bar{\mathcal{A}}_{k}$ and $\bar{b}_{k}$ the rows of $\mathcal{A}_{k}$ and the elements of $b_k$ corresponding to the active constraints.
Then \eqref{opt ineq} becomes
\begin{align*}
\begin{split}
    \hat{d}_{k-1} &= \argmin\limits_d (d-\hat{d}_{k-1}^{u})^\top W^d_{k-1}(d-\hat{d}_{k-1}^{u}) \\
   & \text{subject to}\  \bar{\mathcal{A}}_{k-1} d = \bar{b}_{k-1}.
\end{split}
\end{align*}
The solution of the above program can be found by
\begin{align*}
    \hat{d}_{k-1} &= \hat{d}_{k-1}^{u} - \gamma_{k-1}^d(\bar{\mathcal{A}}_{k-1}\hat{d}_{k-1}^{u} - \bar{b}_{k-1}),
\end{align*}
where $\gamma_{k-1}^d \triangleq (W^d_{k-1})^{-1} \bar{\mathcal{A}}_{k-1}^\top(\bar{\mathcal{A}}_{k-1}  (W^d_{k-1})^{-1} \bar{\mathcal{A}}_{k-1}^\top)^{-1}$.
Its estimation error is
\begin{align}
    \tilde{d}_{k-1} &= (I-\gamma_{k-1}^d\bar{\mathcal{A}}_{k-1})\tilde{d}_{k-1}^{u} + \gamma_{k-1}^d(\bar{\mathcal{A}}_{k-1}d_{k-1}- \bar{b}_{k-1}).
\label{eq:dtile}
\end{align}
The error covariance matrix can be found by
\begin{align}
    P^{d}_{k-1} &\triangleq {\mathbb E}[\tilde{d}_{k-1}\tilde{d}_{k-1}^\top] =(I-\gamma_{k-1}^d\bar{\mathcal{A}}_{k-1})P^{d,u}_{k-1}(I-\gamma_{k-1}^d\bar{\mathcal{A}}_{k-1})^\top,
\end{align}
under the assumption that $\gamma_{k-1}^d(\bar{\mathcal{A}}_{k-1}d_{k-1}- \bar{b}_{k-1})=0$ in~\eqref{eq:dtile}.
The cross error covariance matrix of the state estimate and the actuator attack estimate is
\begin{align*}
    P_{k-1}^{xd} =-P^x_{k-1}A_{k-1}^\top C_{k}^\top M_{k}^\top.
\end{align*}

\subsubsection{Time update}
Given the actuator attack estimate $\hat{d}_{k-1}^u$, the state prediction $\hat{x}_{k|k-1}$ can be updated as in~\eqref{eq:hat x star}. We can derive the error covariance matrix of $\hat{x}^\star_{k|k}$ as
\begin{align}
    P^{\star x}_{k}&\triangleq \E [(\tilde{x}^\star_{k|k} )(\tilde{x}^\star_{k|k})^\top] 
    =A_{k-1}P_{k-1}^xA_{k-1}^\top +A_{k-1}P_{k-1}^{xd}G_{k-1}^\top \nonumber \\
    &+G_{k-1}P_{k-1}^{dx} A_{k-1}^\top + G_{k-1}P_{k-1}^{d}\hat{G}_{k-1}^\top +Q_{k-1} \nonumber \\
    &-G_{k-1}M_{k}C_{k}Q_{k-1}-Q_{k-1}C_{k-1}^\top M_{k}^\top G_{k-1}^\top,
\end{align}
where $P_{k-1}^{dx}= (P_{k-1}^{xd})^\top$.

\subsubsection{Measurement update}
In this step, the measurement $y_{k}$ is used to update the propagated estimate $\hat{x}^\star_{k|k}$ as shown in~\eqref{eq:TU2}.
The covariance matrix of the state estimation error is
\begin{align*}
    P^{x,u}_{k} &\triangleq \E [(\tilde{x}_{k|k}^u)( \tilde{x}_{k|k}^u)^\top]= (I-L_k C_{k})G_{k-1}M_{k}R_{k}L_k^\top +L_k R_{k} L_k^\top \nonumber\\
    &+ L_k R_{k} M_{k}^\top G_{k-1}^\top (I-L_k C_{k})^\top
    +(I-L_k C_{k}) P^{\star x}_{k} (I-L_k C_{k})^\top.
\end{align*}
The gain matrix $L_k$ is chosen by minimizing the trace norm of $P^{x,u}_{k}$: $\min_{L_k} \trace(P^{x,u}_{k})$.
The solution of the program is given by
\begin{align*}
    L_k=(P^{\star x}_{k} C_{k}^\top - G_{k-1} M_{k}  R_{k})\tilde{R}^{\star \dagger}_{k},
\end{align*}
where $\tilde{R}^\star_{k} \triangleq C_{k} P^{\star x}_{k}C_{k}^\top+R_{k}-C_{k}G_{k-1}M_{k}R_{k}-R_{k}M_{k}^\top G_{k-1}^\top C_{k}^\top$.

Now we apply the constraint in~\eqref{eq:2} to the state estimate $\hat{x}_{k|k}^u$.
As Section \ref{Actuator attack estimation}, we formalize the state estimation with the constraints as the constrained convex optimization problem:
\begin{align}
\begin{split}
    \hat{x}_{k|k} &= \argmin\limits_x (x-\hat{x}^u_{k|k})^\top W^x_k(x-\hat{x}^u_{k|k}) \\
   & \text{subject to}\  \mathcal{B}_{k} x \leq c_{k},
\end{split}\label{e035.0}
\end{align}
where we choose $W^x_{k} = (P^{x,u}_{k})^{-1}$ for the smallest error covariance.

We denote by $\bar{\mathcal{B}}_{k}$ and $\bar{c}_{k}$ the rows of $\mathcal{B}_{k}$ and the elements of $c_k$ corresponding to the active constraints of~\eqref{e035.0}. Using the active constraints, we reformulate the problem~\eqref{e035.0} as
\begin{align}
\begin{split}
    \hat{x}_{k|k} &= \argmin\limits_x (x-\hat{x}^u_{k|k})^\top W^x_k(x-\hat{x}^u_{k|k}) \\
   & \text{subject to}\  \bar{\mathcal{B}}_{k} x = \bar{c}_{k}.
\end{split}\label{e035}
\end{align}
The solution of the above problem is given by
\begin{align*}
    \hat{x}_{k|k} &= \hat{x}^u_{k|k} - \gamma_{k}^x(\bar{\mathcal{B}}_{k}\hat{x}^u_{k|k} - \bar{c}_{k}),
\end{align*}
where $\gamma_{k}^x \triangleq (W^x_{k})^{-1} \bar{\mathcal{B}}_{k}^\top(\bar{\mathcal{B}}_{k} (W^x_{k})^{-1}  \bar{\mathcal{B}}_{k}^\top)^{-1}$.
Under the assumption that $\gamma_{k}^x(\bar{\mathcal{B}}_{k}x_k - \bar{c}_{k})=0$ holds, the state estimation error covariance matrix can be expressed as
\begin{align}
    P^{x}_{k} &=\bar{\Gamma}_k P^{x,u}_{k} \bar{\Gamma}_k ^\top, \label{eq:Px projected}
\end{align}
where $\bar{\Gamma}_k \triangleq I-\gamma_{k}^x\bar{\mathcal{B}}_{k}$.

%%%%%%%%%%%%%%%%%%%%%%%%%%%%%%%%%%%%%%%%%%%%%%%%%%%%%%%%%%%
%%%%%%%%%%%%%%%%%%    Analysis      %%%%%%%%%%%%%%%%%%%%%%%
%%%%%%%%%%%%%%%%%%%%%%%%%%%%%%%%%%%%%%%%%%%%%%%%%%%%%%%%%%%
\section{Analysis}\label{analysis}

In Section~\ref{sec:proe}, we show that the projection induced by inequality constraints improves attack-resilient estimation and detection by decreasing the state estimation error and false negative rates.
However, the projection induces a biased estimate as well (Proposition 6 in~\cite{yong2015simultaneous}).
In this context, we will seek to prove practical exponential stability, as shown in Section~\ref{sec:stab}.
All the proofs of Theorems and Lemmas are presented in Appendix for compactness.

%%%%%%%  Performance improvement through constraints  %%%%%%
\subsection{Performance Improvement through Constraints}\label{sec:proe}
The projection reduces the estimation errors and the covariance, as  formulated in Theorem~\ref{the1}.
\begin{theorem}
We have $\|\tilde{x}_{k|k}\|\leq \|\tilde{x}_{k|k}^u\|$ and $\|\tilde{d}_{k}\|\leq \|\tilde{d}_{k}^u\|$; $P_k \leq P_k^u$, and $P_k^d \leq P_k^{d,u}$.
Strict inequality holds if $\rank(\bar{\mathcal{B}}_k) \neq 0$, and $\rank(\bar{\Ac}_k) \neq 0$, respectively.
\label{the1}
\end{theorem}

The properties in Theorem~\ref{the1} are desired for accurate estimation as well as attack detection.
In particular, if the size of attack is smaller than the statistical threshold, the $\chi^2$ detector cannot distinguish the attack from the noise.
Given $d_k \neq 0$, the covariance reduction implies the threshold reduction:
\begin{align*}
d_k^\top (P_{k}^{d,u})^{-1}d_k
&\leq d_k^\top (P_{k}^{d})^{-1}d_k,
\end{align*}
where the test value $d_k^\top (P_{k}^{d})^{-1}d_k$ may reject the null hypothesis, while $d_k^\top (P_{k}^{d,u})^{-1}d_k$ cannot.
Moreover, the estimation error reduction implies an accurate test value:
\begin{align*}
&\|d_k^\top (P_{k})^{-1}d_k-(\hat{d}_k)^\top (P_{k})^{-1}\hat{d}_k\|\\ 
&\leq \|d_k^\top (P_{k})^{-1}d_k-(\hat{d}_k^{u})^\top (P_{k})^{-1}\hat{d}_k^{u}\|,
\end{align*}
which further reduces false negative rates.

%%%%%%%%%%%%%  Stability Analysis  %%%%%%%%%%%%
\subsection{Stability Analysis}\label{sec:stab}
Although the projection reduces the estimation errors and the covariance as shown in Theorem~\ref{the1}, it trades the unbiased estimation off according to Proposition 6 in~\cite{yong2015simultaneous}.
This is because we can guarantee $\bar{\mathcal{B}}_kx_k \leq \bar{c}_k$ instead of $\bar{\mathcal{B}}_kx_k = \bar{c}_k$, but
the unconstrained estimate $\hat{x}_{k|k}^u$ is projected onto $\bar{\mathcal{B}}_kx_k = \bar{c}_k$.
In the absence of the projection ($\tilde{\Ac}_k=0$ and $\bar{\mathcal{B}}_k=0$, $\forall k$), Algorithm~\ref{algorithm1} reduces to the algorithm in~\cite{yong2016unified}, which is unbiased.

It is essential to construct an update law $\tilde{x}_{k|k}$ from $\tilde{x}_{k-1|k-1}$ to analyze stability of the estimation error.
However, the construction is not straight forward comparing to that in filtering with equality constraints~\cite{yong2015simultaneous, simon2002kalman} or filtering without constraints~\cite{yong2016unified, anderson1981detectability}.
Especially, it is difficult to find the exact relation between $\tilde{x}_{k|k}$ and $\tilde{x}_{k|k}^u$:
\begin{align*}
    \tilde{x}_{k|k} &= \tilde{x}^u_{k|k} - \gamma_{k}^x(\bar{\mathcal{B}}_{k}\hat{x}^u_{k|k} - \bar{c}_{k})\neq (I-\gamma_{k}^x\bar{\mathcal{B}}_{k})\tilde{x}^u_{k|k},
\end{align*}
because  $\bar{\mathcal{B}}_kx_k \leq \bar{c}_k$.

To address this issue, we first decompose the estimation error $\tilde{x}_{k|k}$ into two orthogonal spaces
\begin{align}
  \tilde{x}_{k|k} = (I-\gamma_{k}^x\bar{\mathcal{B}}_{k})\tilde{x}_{k|k}+\gamma_{k}^x\bar{\mathcal{B}}_{k}\tilde{x}_{k|k}
  \label{decompose x tilde}
\end{align}
and then, we apply the following lemmas to each term.
\begin{lemma}
    It holds that $(I-\gamma_k^x \bar{\mathcal{B}}_k)\tilde{x}_{k|k}=(I-\gamma_k^x \bar{\mathcal{B}}_k)\tilde{x}_{k|k}^u$.
    \label{lem1}
\end{lemma}
\begin{lemma}
    It holds that $\gamma_k^x\bar{\mathcal{B}}_k\tilde{x}_{k|k}= \alpha_k\gamma_k^x\bar{\mathcal{B}}_k\tilde{x}_{k|k}^u$, where $\alpha_k=\diag{(\alpha_k^1,\cdots,\alpha_k^n)}$ and 
    $\alpha_k^i \triangleq (\gamma_k^x\bar{\mathcal{B}}_k\tilde{x}_k)(i)((\gamma_k^x\bar{\mathcal{B}}_k\tilde{x}_k^u)(i))^\dagger$ $\in [0,1)$ for $i=1,\cdots,n$.
    \label{lem2}
\end{lemma}

According to Lemmas~\ref{lem1} and~\ref{lem2}, the errors in the space $I-\gamma_k^x\bar{\mathcal{B}}_k$ remain identical after the projection, while the errors in the space $\gamma_k^x\bar{\mathcal{B}}_k$ reduce through the projection.
By Lemmas \ref{lem1} and \ref{lem2}, \eqref{decompose x tilde} becomes
\begin{align}
	\tilde{x}_{k|k} = \Gamma_{k}\tilde{x}_{k|k}^u  \label{eq:new c_state error},
\end{align}
where $\Gamma_{k} \triangleq (I-\gamma_k^x \bar{\mathcal{B}}_k)+\alpha_k\gamma_k^x \bar{\mathcal{B}}_k$.
Note that $\alpha_k$ is an unknown matrix and thus cannot be used for the algorithm. We use it only for  analytical purposes.

Now under the following assumptions, we present the stability of Algorithm~\ref{algorithm1}.
\begin{assumption}\label{assum_bounded}
    It holds that $\rank(\mathcal{B}) < n$.
    There exist $\bar{a}$, $\bar{c}_y$, $\bar{g}$, $\bar{m}$, $\ubar{q}$, $\ubar{\beta}$, $\bar{\beta}$ $>0$, such that the following holds for all $k \geq 0$:
    \begin{align*}
        &\|A_k\| \leq \bar{a},
        &&\|C_{k}\| \leq \bar{c}_y,
        &&\|G_{k}\| \leq \bar{g},\\
        & \|M_{k}\| \leq \bar{m},
        &&Q_k\geq \ubar{q} I.
\end{align*}
\end{assumption}
It is assumed that $\rank(\mathcal{B}) < n$; i.e., the number of the state constrains are less than the number of state variables.
The rest of Assumption~\ref{assum_bounded} is widely used in literature on extended KF~\cite{kluge2010stochastic} and nonlinear ISE~\cite{kim2017attack}.
\begin{theorem}\label{the2}
Consider Assumption \ref{assum_bounded} and assume that there exist non-negative constants $\ubar{p}$ and $\bar{p}$ such that
$\ubar{p} I \leq P_k^{x,u} \leq \bar{p} I$ holds for all $k$.
Then the estimation errors $\tilde{x}_{k|k}$ and $\tilde{d}_k$ are practically exponentially stable in mean square; i.e., there exist constants $a_x,a_d,b_x,b_d,c_x,c_d$ such that, for all $k$,
\begin{align*}
    \mathbb{E}[\|\tilde{x}_k\|^2] &\leq a_xe^{-b_xk}\mathbb{E}[\|\tilde{x}_0\|^2] +c_x\nnum\\
    \mathbb{E}[\|\tilde{d}_k\|^2] &\leq a_de^{-b_dk}\mathbb{E}[\|\tilde{d}_0\|^2] +c_d.
\end{align*}
\end{theorem}

Theorem \ref{the2} holds under the assumption of boundedness of $P_k^{x,u}$. One of the sufficient conditions is the uniform detectability of the transformed system as shown in Theorem~\ref{the3}.
\begin{theorem}\label{the3}
If the pair $(C_{k}, \tilde{A}_{k-1})$ is uniformly detectable, then there exist non-negative constants $\ubar{p}$ and $\bar{p}$ such that for all $k$
\begin{align*}
    \ubar{p} I \leq P_k^{x,u} \leq \bar{p} I,
\end{align*}
where $\tilde{A}_{k-1} \triangleq (I-G_{k}M_{k}(C_{k}G_{k-1}M_{k})^{-1}C_{k})\bar{A}_{k-1}\bar{\Gamma}_{k-1}$
and $\bar{A}_{k-1} = (I-G_{k-1}M_{k}C_{k})A_{k-1}$.
\end{theorem}

%%%%%%%%%%%%%%%%%%%%%%%%%%%%%%%%%%%%%%%%%%%%%%%%%%%%%%%%%%%
%%%%%%%%%%%%%%%%%%    Simulation      %%%%%%%%%%%%%%%%%%%%%
%%%%%%%%%%%%%%%%%%%%%%%%%%%%%%%%%%%%%%%%%%%%%%%%%%%%%%%%%%%
\section{Simulation}\label{simulation}
We simulate a scenario shown in Fig. \ref{fig:2agents}, where a multi-agent system that has state and input constraints gets attacked and moves to the attacker's desired place.

\begin{figure}[thpb]
\centering
 \includegraphics[width=0.5\textwidth]{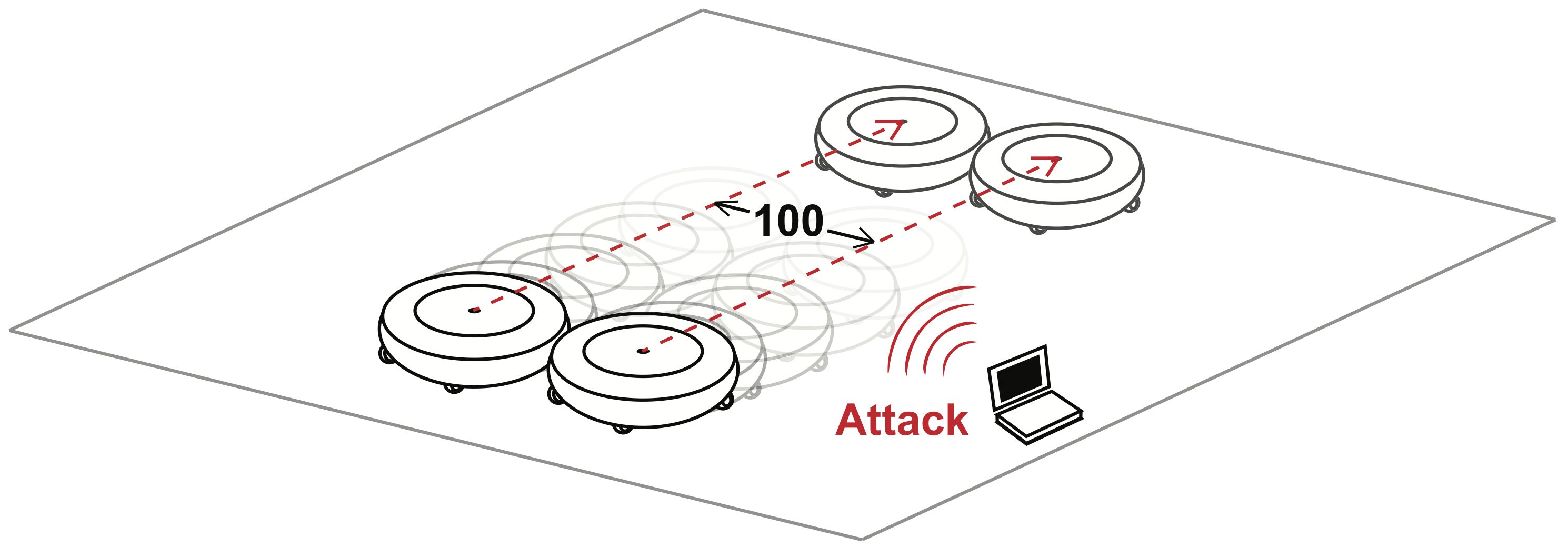}
 \vspace*{-5mm}\caption{Illustration of the simulation scenario: (i) red dash line denoted the path after attack; (ii) 100 denotes the minimum distance difference between two agents by physical state constraint.}
\medskip
\label{fig:2agents}
\end{figure}

%%%%%%%%%%%%%%%%%%  Single Agent Model      %%%%%%%%%%%%%%%
\subsection{Single Agent Model}
We consider a double integrator  dynamic model for each agent
$i \in \{1,\cdots, n\}$, where $n$ denotes the number of agents in the system.
In this simulation, the subscript $(i)$ is used to represent the agent $i$'s vector/matrix; e.g., $x_k^{(i)}$ and $A_k^{(i)}$ denote the state and the system matrix of agent $i$.
Its discrete time state vector $x^{(i)}_k$ that considers planar position and velocity at time step $k$, is given by
\begin{align*}
    x^{(i)}_k = [r^{(i)}_{x,k}, r^{(i)}_{y,k}, v^{(i)}_{x,k}, v^{(i)}_{y,k}]^\top,
\end{align*}
where $r^{(i)}_{x,k}$, $r^{(i)}_{y,k}$ denote $x, y$ position coordinates and $v^{(i)}_{x,k}$, $v^{(i)}_{y,k}$ denote velocity coordinates.
The actuator attack in this simulation is constrained by the acceleration limit, and the state is constrained due to the speed limit and required minimum distance between the two agents:
\begin{align*}
|d^{(i)}_k(n)+u_k^{(i)}(n)| \leq 20, \quad
|v^{(i)}_{x,k}- v^{(i)}_{x,k}|\leq 80;\\
|r^{(i)}_{x,k}-r^{(j)}_{x,k}|\geq 100 \quad \text{or} \quad
|r^{(i)}_{y,k}-r^{(j)}_{y,k}| \geq 100,
\end{align*}
where $(n)$ denotes the $n^{th}$ element in the vector.

Each model is discretized into the following matrices with sampling time of 0.1 seconds:
\begin{equation*}
    A^{(i)}_k =
    \begin{bmatrix} 
    1 & 0 & 0.1 & 0    \\
    0 & 1 & 0    & 0.1 \\
    0 & 0 & 1    & 0 \\
    0 & 0 & 0    & 1
    \end{bmatrix},
    \quad
    B^{(i)}_k = G^{(i)}_k =
    \begin{bmatrix} 
    0    & 0 \\
    0    & 0 \\
    0.1 & 0 \\
    0    & 0.1 
    \end{bmatrix},
\end{equation*}
and the output $y^{(i)}_k$ is the sensor measurement of positions and velocity; i.e., $C^{(i)}_k = I$.
The covariance matrices of noises are chosen as $Q^{(i)}_k = 0.1 I$, and $R^{(i)}_k = 0.01 I$.

%%%%%%%%%%%%%%% Multi-agent System Model      %%%%%%%%%%%%%%%
\subsection{Multi-agent System Model}
The multi-agent system of $n$ agents, where $n \in \N$, can be written in the form of system \eqref{eq:sys2}, where $A_k$ and $C_{k}$ are diagonal matrices as follows: $\diag(A_k) = (A^{(1)}, \cdots, A^{(i)})$, $\diag(C_{k}) = (C_{k}^{(1)}, \cdots, C_{k}^{(i)})$; $B_k = G_k \triangleq [B^{(1)},\cdots, B^{(i)}]^\top$.
The state vector, input vector, actuator attack and sensor measurement are denoted by $x_k \triangleq [x^{(1)}_k, \cdots, x^{(i)}_k]^\top$, $u_k \triangleq [u^{(1)}_k, \cdots, u^{(i)}_k]^\top$, $d_k \triangleq [d^{(1)}_k, \cdots, d^{(i)}_k]^\top$ and $y_{k} \triangleq [y^{(1)}_{k}, \cdots, y^{(i)}_{k}]^\top$, respectively.

\subsection{Attack Scenario}
We consider the scenario that the attacker injects the identical actuator attack to the both agents so that they move horizontally to the right at same time.
The unknown actuator attacks are
\begin{align*}
&d_k(1) = d_k(3) = \left\{ \,
\begin{IEEEeqnarraybox}[][c]{l?s}
\IEEEstrut
20 & if $100n\leq k < 40+100n$, \\
0 & if $40+100n\leq k < 60+100n$, \\
-20 & if $60+100n\leq k < 100+100n$,
\end{IEEEeqnarraybox}
\right.\nnum\\
&d_k(2) = d_k(4)=0,
\end{align*}
where $ n \in \{1, 2, \cdots, 9\}$.

\subsection{Simulation Result}

\begin{figure}[thpb]
\centering
 \includegraphics[width=0.45\textwidth]{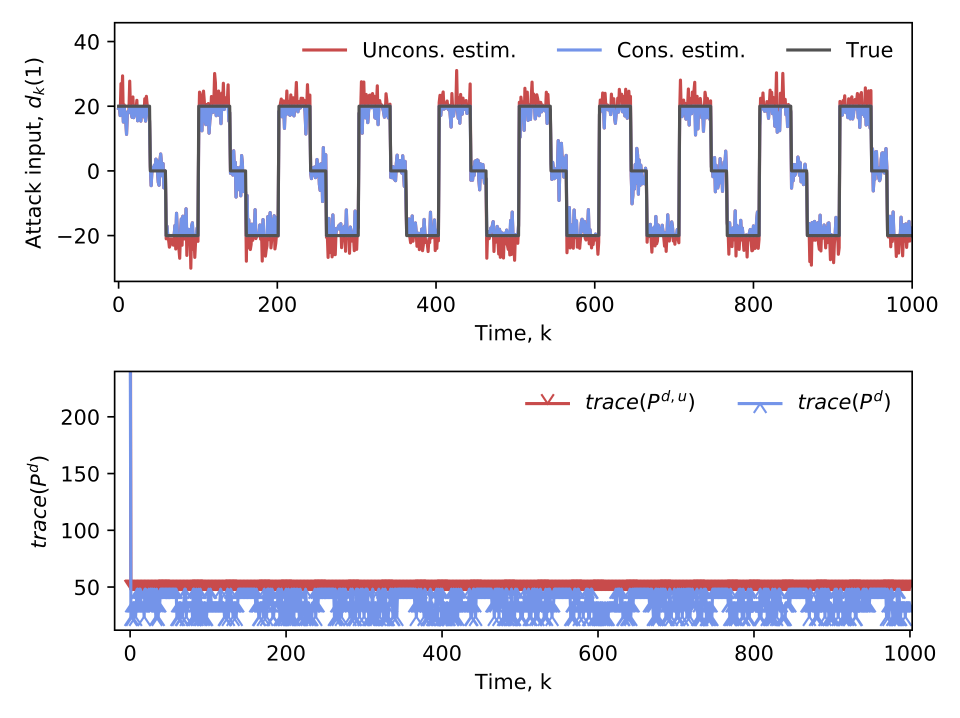}
 \vspace*{-2mm}\caption{Unconstrained and constrained estimation of the actuator attack $d_k(1)$. Trace of unconstrained estimate error covariance of the actuator attack $\trace(P^{d,u})$ and constrained estimate error covariance of the actuator attack $\trace(P^d)$.}
\medskip
\label{fig:da1}
\end{figure}

\begin{figure}[thpb]
\centering
 \includegraphics[width=0.48\textwidth]{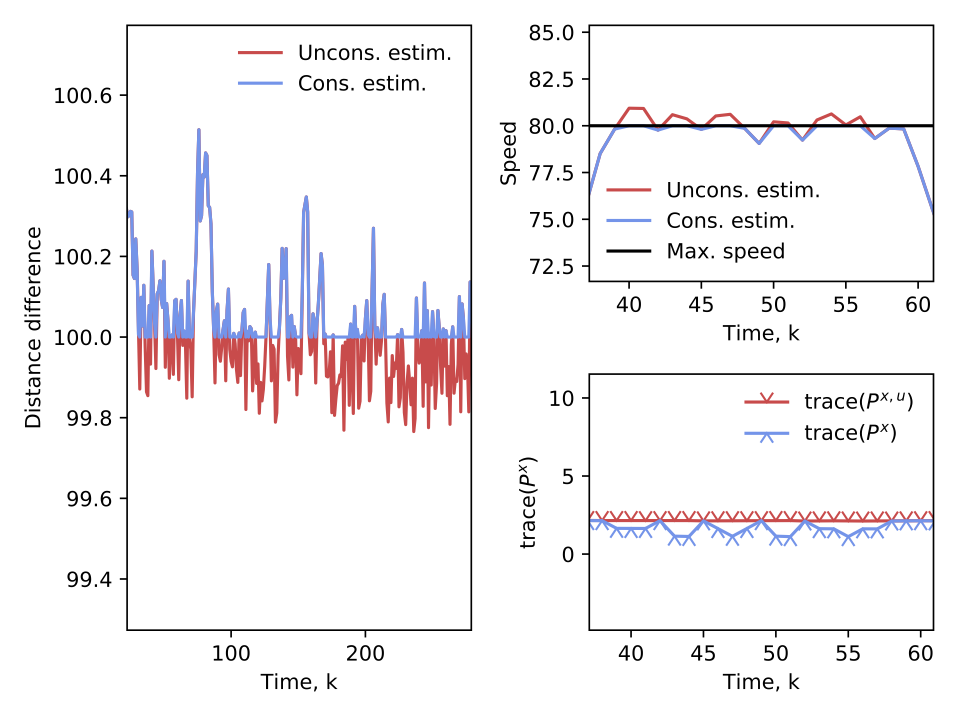}
 \vspace*{-5mm}\caption{Unconstrained and constrained estimation of states (distance difference of two agents and speed of one agent). Trace of unconstrained estimate error covariance of state $\trace(P^{x,u})$ and constrained estimate error covariance of state $\trace(P^x)$.}
\medskip
\label{fig:x32}
\end{figure}

Figures \ref{fig:da1} and \ref{fig:x32} show a comparison of the actuator attack and state estimation with and without the constraints. When the actuator attack estimate and the state estimate are projected to the constrained space, the constrained estimations have smaller estimation error and smaller error covariance as expected.

%%%%%%%%%%%%%%%%%%%%%%%%%%%%%%%%%%%%%%%%%%%%%%%%%%%%%%%%%%%
%%%%%%%%%%%%%%%%%%    Conclusion      %%%%%%%%%%%%%%%%%%%%%
%%%%%%%%%%%%%%%%%%%%%%%%%%%%%%%%%%%%%%%%%%%%%%%%%%%%%%%%%%%

\section{Conclusion} \label{conclusion}
This paper studies attack-resilient estimation algorithm for time-varying stochastic systems given inequality constraints on the states and actuator attacks.
We formally prove that estimation errors and their covariances are less than those from unconstrained algorithms, which is a desired condition for attack detection in stochastic systems.
We prove that the estimation errors are practically exponentially stable.
A simulation is presented to reveal the attack-resilient property and efficiency of the proposed algorithm in attack detection.

%%%%%%%%%%%%%%%%%%%%%%%%%%%%%%%%%%%%%%%%%%%%%%%%%%%%%%%%%%%
%%%%%%%%%%%%%%%%%%    References     %%%%%%%%%%%%%%%%%%%%%%
%%%%%%%%%%%%%%%%%%%%%%%%%%%%%%%%%%%%%%%%%%%%%%%%%%%%%%%%%%%
\bibliography{references}{}
\bibliographystyle{ieeetr}

%%%%%%%%%%%%%%%%%%%%%%%%%%%%%%%%%%%%%%%%%%%%%%%%%%%%%%%%%%%
%%%%%%%%%%%%%%%%%%    Appendix     %%%%%%%%%%%%%%%%%%%%%%%%
%%%%%%%%%%%%%%%%%%%%%%%%%%%%%%%%%%%%%%%%%%%%%%%%%%%%%%%%%%%
\appendix \label{appendix}

\subsection{Proof of Lemma~\ref{lem1}}
    The relationship can be obtained by applying $\bar{\mathcal{B}}_k\hat{x}_{k|k}=\bar{c}_k$ to
\begin{align*}
\tilde{x}_{k|k} &= x_k - \hat{x}_{k|k} = x_k - (\hat{x}_{k|k}^u-\gamma_k^x(\bar{\mathcal{B}}_k\hat{x}_{k|k}^u-\bar{c}_k))\nnum\\
&=\tilde{x}_{k|k}^u+\gamma_k^x(\bar{\mathcal{B}}_k\hat{x}_{k|k}^u-\bar{c}_k)\nnum\\
&=\tilde{x}_{k|k}^u+\gamma_k^x(\bar{\mathcal{B}}_k\hat{x}_{k|k}^u-\bar{\mathcal{B}}_k\hat{x}_{k|k})\nnum\\
&=\tilde{x}_{k|k}^u-\gamma_k^x(\bar{\mathcal{B}}_k\tilde{x}_{k|k}^u-\bar{\mathcal{B}}_k\tilde{x}_{k|k}),
\end{align*}
which implies the statement.
\hfill $\blacksquare$

\subsection{Proof of Lemma~\ref{lem2}}
The solution of $\bar{\mathcal{B}}_k x \leq \bar{c}_k$ defines a closed convex set ${\mathcal C}_k$. The point $\hat{x}_{k|k}^u$ is not an element of the convex set.
The point $\hat{x}_{k|k}$ has the minimum distance from $\hat{x}_{k|k}^u$ with metric $d(a,b)=\|a-b\|_{W_k^x}$ in the convex set ${\mathcal C}_k$ by~\eqref{e035}.
Since the solution $x_k$ is in the closed set ${\mathcal C}_k$, and $\gamma_k^x\bar{\mathcal{B}}_k$ is a weighted projection with weigh $W_k^x$, the statement holds. 
\hfill $\blacksquare$
\subsection{Proof of Theorem~\ref{the1}}

The statement for $\|\tilde{x}_{k|k}\|\leq \|\tilde{x}_{k|k}^u\|$ and
$\|\tilde{d}_{k}\|\leq \|\tilde{d}_{k}^u\|$
is the direct result of Lemmas~\ref{lem1} and~\ref{lem2}.

To show the rest of properties, we first identify the equality
\begin{align}
    (I-\gamma_k^x \bar{\mathcal{B}}_k)^\top\gamma_k^x \bar{\mathcal{B}}_k=0.
    \label{cle2}
\end{align}
Since $\bar{\mathcal{B}}_k\gamma_k^x=I$, it holds that 
$\gamma_k^x\bar{\mathcal{B}}_k\gamma_k^x = \gamma_k^x$, and $\bar{\mathcal{B}}_k\gamma_k^x\bar{\mathcal{B}}_k = \bar{\mathcal{B}}_k$; i.e., $\gamma_k^x = \bar{\mathcal{B}}_k^\dagger$.
Then, we have $\bar{\mathcal{B}}_k^\top (\gamma_k^x)^\top \gamma_k^x=\gamma_k^x$, which implies
$
\tilde{x}_{k|k}^\top(I-\gamma_k^x \bar{\mathcal{B}}_k)^\top\gamma_k^x \bar{\mathcal{B}}_k\tilde{x}_{k|k}=\tilde{x}_{k|k}^\top(\gamma_k^x\bar{\mathcal{B}}_k-\bar{\mathcal{B}}_k^\top (\gamma_k^x)^\top \gamma_k^x\bar{\mathcal{B}}_k)\hat{x}_{k|k}=0$.
Note that~\eqref{cle2} holds for $(\tilde{x}_{k|k}^u)^\top(I-\gamma_k^x \bar{\mathcal{B}}_k)^\top\gamma_k^x \bar{\mathcal{B}}_k\tilde{x}_{k|k}^u=0$ as well.

Inequalities for the covariance can be obtained by taking the trace
\begin{align*}
    \trace{(P_k^x)} &= \trace{((I-\gamma_k^x \bar{\mathcal{B}}_k)P_k^x(I-\gamma_k^x \bar{\mathcal{B}}_k)^\top)}\nnum\\
    &=\trace{((I-\gamma_k^x \bar{\mathcal{B}}_k)^\top(I-\gamma_k^x \bar{\mathcal{B}}_k)P_k^x)}\nnum\\
    &=\trace{(I-\gamma_k^x \bar{\mathcal{B}}_k)P_k^x)}\nnum\\
    &=\trace{(P_k^x)}-\trace{(\gamma_k^x \bar{\mathcal{B}}_kP_k^x)},
\end{align*}
where~\eqref{cle2} has been applied.
Because $\gamma_k^x \bar{\mathcal{B}}_kP_k^x=(P_k^x)^{-1} \bar{\mathcal{B}}_{k}^\top(\bar{\mathcal{B}}_{k} (P_k^x)^{-1}  \bar{\mathcal{B}}_{k}^\top)^{-1}\bar{\mathcal{B}}_kP_k^x>0$, we have the desired result. The same relation for $P_k^d$ can be obtained by the same procedure.
\hfill $\blacksquare$

\subsection{Proof of Theorem~\ref{the3}}
The unconstrained state estimation error can be found by
\begin{align}
    \tilde{x}^u_{k|k} &= (I-L_kC_{k}) \bar{A}_{k-1}\tilde{x}_{k-1|k-1} 
    +(I-L_kC_{k})\bar{w}_{k-1}+\bar{L}_kv_{k},
 \label{eq:new unc_state error}
\end{align}
where $\bar{w}_{k-1} \triangleq (I- G_{k-1}M_{k}C_{k})w_{k-1}$, and $\bar{L}_k \triangleq L_kC_{k}G_{k-1}M_{k}-L_k-G_{k-1}M_{k}$. 
Therefore, the update law of unconstrained covariance matrix is calculated from~\eqref{eq:new unc_state error} and \eqref{eq:Px projected}:
    \begin{align}
        P_{k}^{x,u} &=(I-L_kC_{k})\bar{A}_{k-1}\bar{\Gamma}_{k-1} P_{k-1}^{x,u} \bar{\Gamma}_{k-1}^\top\bar{A}_{k-1}^ \top(I-L_kC_{k})^\top \nnum\\
        &+\bar{L}_kR_{k}\bar{L}_k^\top +(I-L_kC_{k})\bar{Q}_{k-1}(I-L_kC_{k})^\top.
        \label{eq: new P (Gamma)}
    \end{align}
Covariance update law~\eqref{eq: new P (Gamma)} is identical to the covariance update law of the KF solution of the transformed system:
\begin{align}
    x_{k} &= \bar{A}_{k-1}\bar{\Gamma}_{k-1}x_{k-1}+\hat{w}_{k-1}\\
    y_{k} &= C_{k}x_{k}+v_{k},\label{eq:equiv1}
\end{align}
where $\hat{w}_{k-1} \triangleq - G_{k-1}M_{k}C_{k}w_{k-1} - G_{k-1}M_{k}v_{k}+w_{k-1}$.
However, in the transformed system, the process noise and measurement noise are correlated; i.e., $\mathbb{E}[\hat{w}_{k-1}v_{k}^\top]=-G_{k-1}M_{k}R_{k}\neq 0$. To decouple the noises, we add a zero term $K_k(y_{k}-C_{k}(\bar{A}_{k-1}\bar{\Gamma}_{k-1}x_{k}+\hat{w}_{k-1})-v_{k})$ to the state equation to obtain:
\begin{align*}
    x_{k} &= \tilde{A}_{k-1}x_{k-1}+\tilde{u}_{k-1}+ \tilde{w}_{k-1},
\end{align*}
where $\tilde{A}_{k-1} \triangleq (I-K_kC_{k})\bar{A}_{k-1}\bar{\Gamma}_{k-1}$, $\tilde{u}_{k-1} \triangleq K_ky_{k}$ is the known input, and $\tilde{w}_{k-1} \triangleq (I-K_kC_{k})\hat{w}_{k-1}-K_kv_{k}$ is the new process noise. The new process noise and the measurement noise could be decoupled by choosing the gain $K_k$ such that
\begin{align*}
    \mathbb{E}[\tilde{w}_{k-1}v_{k}^\top]=0.
\end{align*}
The solution can be found by $K_k = G_{k}M_{k}(C_{k}G_{k-1}M_{k})^{-1}$.
Then, the system~\eqref{eq:equiv1} becomes 
\begin{align*}
    x_{k+1} &= \tilde{A}_{k}x_{k}+\tilde{u}_{k}+\tilde{w}_{k}\\
    y_{k} &= C_{k}x_{k}+v_{k}.
\end{align*}
Since the pair $(C_{k},\tilde{A}_{k})$ is uniformly detectable, by Corollary 5.2 in~\cite{anderson1981detectability}, the statement holds. 
\hfill $\blacksquare$
\subsection{Proof of Theorem~\ref{the2}}
Consider the Lyapunov function
\begin{align*}
    V_k = (\tilde{x}^u_{k|k})^\top (P_k^{x,u})^{-1}(\tilde{x}^u_{k|k}).
\end{align*}
After substituting  \eqref{eq:new unc_state error} into the Lyapunov function, we obtain
\begin{align}\label{lyap_vk}
    V_k &= (\tilde{x}^u_{k-1|k-1})^\top\Gamma_{k-1}^\top\bar{A}_{k-1}^\top(I-L_kC_{k})^\top(P_k^{x,u})^{-1}\nnum\\
    & \times (I-L_kC_{k})\bar{A}_{k-1}\Gamma_{k-1}\tilde{x}^u_{k-1|k-1}\nnum\\
    &+2 (\tilde{x}^u_{k-1|k-1})^\top\Gamma_{k-1}^\top\bar{A}_{k-1}^\top(I-L_kC_{k})^\top\nnum\\
    &\times (P_k^{x,u})^{-1}(I-L_kC_{k})\bar{w}_{k-1}\nnum\\
    &+ 2(\tilde{x}^u_{k-1|k-1})^\top\Gamma_{k-1}^\top\bar{A}_{k-1}^\top(I-L_kC_{k})^\top(P_k^{x,u})^{-1}\bar{L}_{k}v_{k}\nnum\\
    &+ \bar{w}_{k-1}^\top(I-L_kC_{2,k})^\top(P_k^{x,u})^{-1}(I-L_kC_{k})\bar{w}_{k-1}\nnum\\
    &+ 2w_{k-1}^\top(I-L_kC_{k})^\top(P_k^{x,u})^{-1}\bar{L}_{k}v_{k}+ v_{k}^\top \bar{L}_{k}(P_k^{x,u})^{-1}\bar{L}_{k}v_{k}.
\end{align}
By the uncorrelatedness property~\cite{papoulis2002probability} of $w_{k-1}$, $v_k$ and $\tilde{x}^u_{k-1|k-1}$, the Lyapunov function \eqref{lyap_vk} becomes
\begin{align}
    \mathbb{E}[V_k] &= \mathbb{E}[(\tilde{x}_{k-1|k-1}^u)^\top\Gamma_{k-1}^ \top\bar{A}_{k-1}^\top(I-L_kC_{k})^\top(P_k^{x,u})^{-1}\nnum\\
    &\times \bar{A}_{k-1}(I-L_kC_{k})\Gamma_{k-1}(\tilde{x}_{k-1|k-1}^u)]\nnum\\
    &+ \mathbb{E}[\bar{w}_{k-1}^\top(I-L_kC_{k})^\top(P_k^{x,u})^{-1}(I-L_kC_{k})\bar{w}_{k-1}]\nnum\\
    &+ \mathbb{E}[v_{k}^\top \bar{L}_{k}(P_k^{x,u})^{-1}\bar{L}_{k}v_{k}].
    \label{Lyap:exp}
\end{align}
The following  statements are formulated to deal with each term in~\eqref{Lyap:exp}.
\begin{claim}
    There exists a constant $\delta \triangleq (\frac{\ubar{q}'}{\bar{a}'^2\bar{p}}+1)^{-1} \in (0,1)$, such that $\Gamma_{k-1}^\top\bar{A}_{k-1}^\top(I-L_kC_{k})^\top (P_k^{x,u})^{-1}(I-L_kC_{k})\bar{A}_{k-1}\Gamma_{k-1} < \delta (P_{k-1}^{x,u})^{-1}$.\\
    \label{cl3}
\end{claim}
\begin{proof}
    Since $\rank(\Bc)<n$, it holds that $\rank(\bar{\Bc})<n$ and thus $\bar{\Gamma} \neq 0$. Therefore, $\|\bar{\Gamma}_{k-1}\| = 1$ because $\gamma_{k-1}^x\bar{\mathcal{B}}_{k-1}$ is a projection matrix.
    From Assumption \ref{assum_bounded} and Theorem \ref{the3}, we have
    \begin{align*}
        &\bar{Q}_{k-1} \geq \ubar{q}'I,
        &&P_{k-1}^{x} \leq \bar{p}I.
        %&&\|\bar{\Gamma}_{k-1}\| =1
    \end{align*}
    Since $\|\bar{A}_{k-1}\|$ is upper bounded by $\bar{a}' \triangleq \bar{a}(1+\bar{g}\bar{m}\bar{c}_y)$, we can have $\bar{A}_{k-1}\bar{A}_{k-1}^\top \leq \bar{a}'^2I$. Then
    \begin{align}
        \bar{Q}_{k-1} &\geq \ubar{q}' \frac{\bar{A}_{k-1}\bar{A}_{k-1}^\top}{\bar{a}'^2}
        \geq  \frac{\ubar{q}'}{\bar{a}'^2}\bar{A}_{k-1}\bar{\Gamma}_{k-1}\bar{\Gamma}_{k-1}^ \top\bar{A}_{k-1}^\top\nnum\\
        &\geq \frac{\ubar{q}'}{\bar{a}'^2\bar{p}} \bar{A}_{k-1}\bar{\Gamma}_{k-1} P_{k-1}^{x,u} \bar{\Gamma}_{k-1}^ \top\bar{A}_{k-1}^\top.
        \label{eq:qbar}
    \end{align}
    Substitution of~\eqref{eq:qbar} into \eqref{eq: new P (Gamma)} yields
    \begin{align}\label{eq:ineq_in_proof}
         &P_k^{x,u}- (1+\frac{\ubar{q}'}{\bar{a}'^2\bar{p}})(I-L_kC_{k})\bar{A}_{k-1}\bar{\Gamma}_{k-1}P_{k-1}^{x,u}\bar{\Gamma}_{k-1}^\top\bar{A}_{k-1}^\top \nnum\\
        &\times (I-L_kC_{k})^\top>0,
    \end{align}
    where the inequality holds because $R>0$.
     As $(1+\frac{\ubar{q}'}{\bar{a}'^2\bar{p}})P_{k-1}^{x,u} >0$, the inverse of the left hand side of \eqref{eq:ineq_in_proof} exists and is symmetric positive definite. By the matrix inversion lemma~\cite{tylavsky1986generalization}, it follows that
    \begin{align}
        &(1+\frac{\ubar{q}'}{\bar{a}'^2\bar{p}})^{-1} (P_{k-1}^{x,u})^{-1} - \bar{\Gamma}_{k-1}^ \top\bar{A}_{k-1}^\top(I-L_kC_{k})^\top(P_k^{x,u})^{-1}\nnum\\
        &\times (I-L_kC_{k})\bar{A}_{k-1}\bar{\Gamma}_{k-1} >0.
        \label{cleq:000}
    \end{align}
    Since $\gamma_{k-1}^x\bar{\mathcal{B}}_{k-1}$ is a positive definite matrix, and $\|\alpha_{k-1}\|\leq 1$, we have
    \begin{align*}
        I-\gamma_{k-1}^x\bar{\mathcal{B}}_{k-1}\leq\Gamma_{k-1}=I-\gamma_{k-1}^x\bar{\mathcal{B}}_{k-1}+\alpha_{k-1}\gamma_{k-1}^x\bar{\mathcal{B}}_{k-1}\leq I,
    \end{align*}
    which implies $\|\Gamma_{k-1}\|\leq1$.
    Since $\|\bar{\Gamma}_{k-1}\| = 1$ and $\|\Gamma_{k-1}\|\leq1$, inequality~\eqref{cleq:000} proves the claim.
\end{proof}

\begin{claim}\label{cl4}
    There exists a positive constant $c \triangleq \bar{p}(1+\bar{l}\bar{c}_y)^2(1+\bar{g}\bar{m}\bar{c}_{2})^2\bar{q}\ \rank(Q_{k-1}) + \bar{p}(\bar{l}\bar{c}_y\bar{g}\bar{m}-\bar{l}-\bar{g}\bar{m})^2\bar{r_2}\ \rank(R_{2,k})$, such that
    \begin{align*}
        \mathbb{E}[\|(I-L_kC_{k})^\top(P_k^{x,u})^{-1}(I-L_kC_{k})\| \|\bar{w}_{k-1}\|^2]\\
        +\mathbb{E}[\|\bar{L}_{k}(P_k^{x,u})^{-1}\bar{L}_{k}\| \|v_{k}]\|^2]\leq c.\\
    \end{align*}
   \end{claim}
    \begin{proof}
        The first term is bounded by:
        \begin{align*}
            \mathbb{E}&[\|(I-L_kC_{k})^\top(P_k^{x,u})^{-1}(I-L_kC_{k})\| \|\bar{w}_{k-1}\|^2]\\
            =&\mathbb{E}[\|(I-L_kC_{k})^\top(P_k^{x,u})^{-1}(I-L_kC_{k})\|\\
            &\|(I- G_{k-1}M_{k}C_{k})\|^2\|w_{k-1}\|^2]\\
            &\leq \bar{p}(1+\bar{l}\bar{c}_y)^2(1+\bar{g}\bar{m}\bar{c}_{2})^2\bar{q}\ \rank(Q_{k-1}),
              \end{align*}
              where we apply $\|w_{k-1}\|^2 = \trace(w_{k-1}w_{k-1}^\top)\leq \bar{q}\ \rank(Q_{k-1})$.
            Likewise, the second term is bounded by:
              \begin{align*}
             \mathbb{E}[\|\bar{L}_{k}(P_k^{x,u})^{-1}\bar{L}_{k}\| \|v_{k}]\|^2] \leq \bar{p}(\bar{l}\bar{c}_y\bar{g}\bar{m}+\bar{l}+\bar{g}\bar{m})^2\bar{r_2}\ \rank(R_{k}).
        		\end{align*}
        		These complete the proof.
    \end{proof}

Through Claims~\ref{cl3} and~\ref{cl4},~\eqref{Lyap:exp} becomes
\begin{align*}
    \mathbb{E}[V_k] &\leq \delta \mathbb{E}[V_{k-1}] + c.
\end{align*}
By recursively applying the above relation, we have
\begin{align*}
    \mathbb{E}[V_k] &\leq \delta^k\mathbb{E}[V_0] + \sum_{i=0}^{k-1}\delta^i c \leq \delta^k\mathbb{E}[V_0] + \sum_{i=0}^{\infty}\delta^i c\nnum\\
    & = \delta^k\mathbb{E}[V_0] + \frac{c}{1-\delta},
\end{align*}
which implies practical exponential stability of the estimation error:
\begin{align*}
    \mathbb{E}[\|\tilde{x}_k^u\|^2] &\leq \frac{\bar{p}}{\ubar{p}}\delta^k\mathbb{E}[\|\tilde{x}_0^u\|^2]  +\frac{c\bar{p}}{(1-\delta)}\nnum\\
    &= a_x'e^{-b_x'k}\mathbb{E}[\|\tilde{x}_0^u\|^2] +c_x',
\end{align*}
where $(\tilde{x}_{k|k}^u)^\top (P_k^x)^{-1} (\tilde{x}_{k|k}^u) \geq \lambda_{\min}{((P_k^x)^{-1})}\|\tilde{x}_{k|k}^u\|^2 \geq \frac{1}{\bar{p}}\|\tilde{x}_{k|k}^u\|^2$
and
$(\tilde{x}_{0|0}^u)^\top (P_0^x)^{-1} \tilde{x}_{0|0}^u \leq \lambda_{\max}{((P_0^x)^{-1})}\|\tilde{x}_{0|0}^u\|^2 \leq \frac{1}{\ubar{p}}\|\tilde{x}_{0|0}^u\|^2$ have been applied.
Constants are defined by
\begin{align*}
    &a_x'\triangleq \frac{\bar{p}}{\ubar{p}},
    &&b_x'\triangleq \ln(1+\frac{\ubar{q}'}{\bar{h}^2\bar{a}'^2\bar{p}})
    &&c_x'\triangleq \frac{c\bar{p}}{(1-\delta)}.
\end{align*}
Since $\tilde{x}_{k|k}$ is a linear transformation of $\tilde{x}_{k|k}^u$, the same stability holds for $\tilde{x}_{k|k}$.
Likewise, the same stability holds for $\tilde{d}_k$ in~\eqref{eq:dtile} because it is a linear transformation of $\tilde{x}_{k|k}$. We omit its details.
\hfill $\blacksquare$

%%%%%%%%%%%%%%%%%%%%%%%%%%%%%%%%%%%%%%%%%%%%%%%%%%%%%%%%%%%%%%%%%%%%%%%%%%%%%%%%

\end{document}